\documentclass[12pt]{amsart}
\input{commands.tex}
\usepackage{float}
\title{Bitangents to symmetric quartics}

\author{Candace Bethea and Thomas Brazelton}

\newcommand{\linecolorone}{green}
\newcommand{\linecolortwo}{pink}
\newcommand{\linecolorthree}{blue}
\newcommand{\linecolorfour}{orange}

\newcommand{\quarticpicwidth}{0.4\linewidth}
\let\genfontsize\tiny

\newcommand{\gpdiagram}[2]{{\color{gray}\mathrm{#1}}\,|\,#2}

\begin{document}
\maketitle

\begin{abstract}
    Recall that a non-singular planar quartic is a canonically embedded non-hyperelliptic curve of genus three. We say such a curve is \textit{symmetric} if it admits non-trivial automorphisms. The classification of (necessarily finite) groups appearing as automorphism groups of non-singular curves of genus three dates back to the last decade of the 19th century. As these groups act on the quartic via projective linear transformations, they induce symmetries on the 28 bitangents. Given such an automorphism group $G=\Aut(C)$, we prove the $G$-orbits of the bitangents are independent of the choice of $C$, and we compute them for all twelve types of smooth symmetric planar quartic curves. We further observe that techniques deriving from equivariant homotopy theory directly reveal patterns which are not obvious from a classical moduli perspective.
\end{abstract}

\section{Introduction}

The study of bitangents to plane quartics dates back to the 19th century, and has been one of the pillars of classical enumerative algebraic geometry. In 1850, Jacobi \cite{D315} was the first to compute that the number of bitangents to a smooth planar quartic was 28, using work of Pl\"{u}cker from two decades earlier \cite{D442}. This was extended to general quartics shortly thereafter by Hesse \cite{D288}. The configurations of these lines, their combinatorics, and their sweeping connections throughout mathematics were objects of deep fascination throughout the 19th century, and their connections to other fields of geometry, e.g., Lie algebras, continue to be topics of interest to modern mathematicians \cite{Manivel}. The automorphism groups of genus three (non-hyperelliptic) curves have been completely classified, and we provide a historical overview of this work in \autoref{subsec:aut-quartics}.
Since a planar quartic is a canonically embedded non-hyperelliptic genus three curve, any automorphism of a genus three curve will extend to an automorphism of the projective plane, and such an automorphism group will be finite by the Hurwitz formula. In particular the bitangents of such a curve will inherit a natural action from the automorphism group of the curve. We prove the action on the bitangents is independent of the choice of symmetric quartic, and we compute it in each case.

\begin{theorem}\label{thm:main} 
Let $G$ be any group which appears as the automorphism group of a genus three non-hyperelliptic curve. Then for any such curve $C$ with $G \cong \Aut(C)$, the action of $G$ on the bitangents of $C$ is independent of the choice of $C$, and is given by the following $G$-set:\footnote{See \autoref{section:computations} for an explanation of the notation.}
\begin{center}
    \begin{tabular}{|r | c | l || l|}
    \hline
    Type & $\Aut(C)$  & Order &  28 bitangents \\
    \hline
    \hline
    \hyperref[subsec:Klein]{I} & $\PSL_2(7)$ & 168 &  $[\PSL_2(7)/S_3]$ \\
    \hyperref[subsec:Dyck]{II} &  $C_4^{\times 2} \rtimes S_3$ & 96 &$[C_4^{\times 2} \rtimes S_3/C_6] + [C_4^{\times 2} \rtimes S_3/C_8]$ \\
    \hyperref[subsec:III]{III} & $C_4\circledcirc A_4$  & 48 & $[C_4 \circledcirc A_4 / C_2] + [C_4 \circledcirc A_4 / C_{12}]$ \\
    \hyperref[subsec:IV]{IV} & $S_4$ & 24 & $[S_4/C_2^o] + [S_4/C_2^e] + [S_4/S_3]$ \\
    \hyperref[subsec:V]{V} & $P$ & 16 &  $[P/C_2^{(1)}] + [P/C_2^{(2)}] + [P/C_2^{(3)}] + [P/C_4^Z]$ \\
    \hyperref[subsec:VI]{VI} & $C_9$ & 9 &  $3[C_9/e] + [C_9/C_9]$ \\
    \hyperref[subsec:VII]{VII} & $D_8$ & 8 &  $[D_8/e] + 2[D_8/C_2^{(1)}] + 2[D_8/C_2^{(2)}] + [D_8/C_2^Z]$ \\
    \hyperref[subsec:VIII]{VIII} & $C_6$ & 6 &  $4[C_6/e] + [C_6/C_2] + [C_6/C_6]$ \\  
    \hyperref[subsec:IX]{IX} & $S_3$  & 6 &  $3[S_3/e] + 3[S_3/C_2] + [S_3/S_3]$ \\
    \hyperref[subsec:X]{X} & $K_4$ & 4 &  $5[K_4/e] + 2[K_4/C_2^L] + 2[K_4/C_2^R]$ \\    
    \hyperref[subsec:XI]{XI} & $C_3$ & 3 &  $9[C_3/e] + [C_3/C_3]$ \\
    \hyperref[subsec:XII]{XII} & $C_2$ & 2 &  $12[C_2/e] + 4[C_2/C_2]$ \\
    \hline
    \end{tabular}
\end{center}
\end{theorem}

\begin{example}\label{exa:edge-D8} 
As an example, we visualize the 28 bitangents on the following \textit{Edge quartic}.\footnote{%
See \autoref{rmk:edge} for a further discussion of Edge quartics.
}
This admits a clear $D_8$ symmetry on the $xy$-plane, which extends to an $S_4$-symmetry by passing between the different affine patches.
\begin{figure}[h]
  \includegraphics[width=0.4\linewidth]{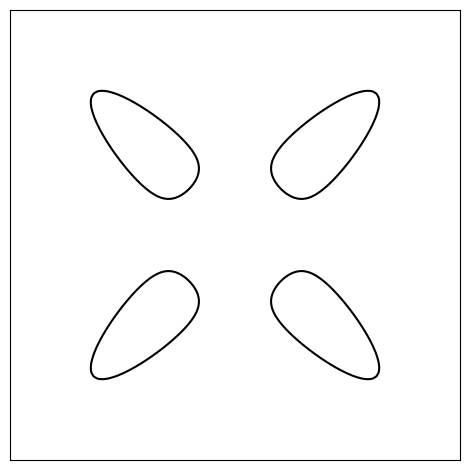}
   \includegraphics[width=0.4\linewidth]{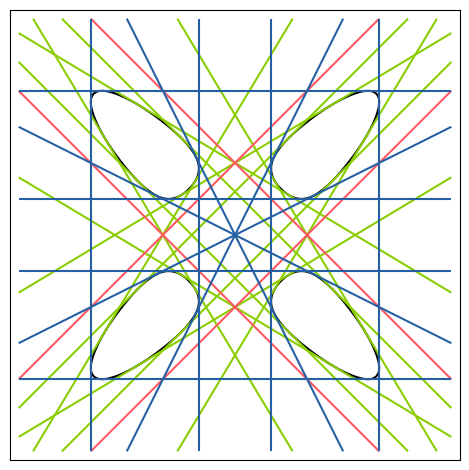}
  \centering
  \caption{The \textit{Edge quartic}, defined by the equation $25(x^4 + y^4 + z^4) -34(x^2y^2 + x^2z^2 + y^2z^2)$, and its 28 real bitangents graphed with their $S_4$ orbits.}
  \label{fig:edge-d8}
\end{figure}
\end{example}

We prove \autoref{thm:main} by leveraging the machinery of \cite{EEG} by first formulating the count of bitangents on a $G$-symmetric quartic as a $G$-equivariant vector bundle, and computing its $\MU_G$-valued Euler class --- this proves that the action of $G\cong\Aut(C)$ on the bitangents is independent of $C$. We pick candidate curves in each automorphism class from \cite{Dolgachev,Henn} and compute equations for each of the bitangents using code modified from \cite{PSV}. After this we determine the orbits of lines and their isotropy under the action of the automorphism group. We provide figures in each case, and discuss rationality where it is relevant.

\begin{remark} \textit{(On related literature)}
\begin{itemize}
    \item The orbits of the automorphism group of the Klein quartic on its 28 bitangents is well-studied, and visualizable over $\P^1_{\mathbb{F}_7}$ (\cite{Dolg134,Dolg318}). In general the equivariant enumerative geometry of the Klein quartic has been a popular direction of research since the mid-to-late 1800's --- the action on the even theta characteristics is in \cite{Dolgachev-Kanev}, and the action on the Steiner complexes and Aronhold sets in \cite{Dolg318}. See \cite[Remark~6.5.4]{Dolgachev} for more information.
    \item The bitangents  of non-hyperelliptic curves of genus three with cyclic automorphisms of degree 3, 6, or 9 was recently studied in a paper of Liang \cite{Liang}.
\end{itemize}
\end{remark}

\begin{remark} \textit{(On rationality)} In 1873, Zeuthen proved that a smooth real quartic can have 3, 8, 16, or 28 bitangents, and this depends precisely on the topology of the real curve \cite{Zeuthen}. One might ask whether results in this paper can make general statements about the number of real bitangents to smooth real planar quartics with given automorphism group, in a way analogous to \cite[Theorem~1.3]{EEG}. When the automorphism group is real (lies in $\PGL_3(\R)$), orbits of odd size necessarily are defined over the reals. Beyond this, we cannot say much as there isn't a clear connection between automorphism types (in the decomposition in \autoref{thm:main}) and the real topology of a curve ---  we can observe this explicitly by traveling through a family of curves of a given type and seeing that the number of real bitangents can vary wildly. For example a Type XII quartic can have no real points, two non-nested ovals, two nested ovals, or three ovals. 
\end{remark}

\subsection{Outline} The paper is organized as follows. Section \ref{section:background} gives background information on canonical quartic plane curves which groups can arise as non-trivial automorphisms of them. Section \ref{section:euler} defines the vector bundle which parameterizes the enumerative problem of counting bitangents to symmetric quartics, following the construction of \cite{LarsonVogt}, shows it's an equivariant vector bundle for any group $G$ which arises as a non-trivial automorphism group of a canonical quartic, and leverages the equivariant Euler number of \cite{EEG} to conclude that the equivariant count of bitangents on a $G$-symmetric quartic is independent of the choice of quartic. Section \ref{section:computations} gives example quartics for each group and graphs orbits of real bitangents for each. Section \ref{sec:subgroups} concludes the paper with a subgroup lattice of automorphism groups of canonical quartics. 

\subsection{Data availability}

Throughout the paper are various \textbf{Computation} environments. These refer to numerical solutions for bitangents to example quartic curves for each type, whose orbits under the automorphism group of the curve were then checked on a computer (but can also be verified by hand). Explicit equations for all the lines for the example quartics referenced, grouped into their orbits, as well as the code the authors used to compute bitangents to planar quartics is available here:
\begin{center}
    \url{https://tbrazel.github.io/supplementary/lines.pdf}
\end{center}

\subsection{Acknowledgements} The second named author would like to thank Joe Harris for helpful conversations related to this work. Candace Bethea was supported by the National Science Foundation
award DMS-240209 for the duration of this work. Thomas Brazelton was supported by an NSF Postdoctoral Research Fellowship (DMS-2303242).

\section{Background}\label{section:background}

\subsection{Quartics, the basics}
Nonsingular curves of genus $g\ge2$ come in two flavors, being \textit{canonical} and \textit{hyperelliptic}, depending on whether the canonical class is very ample. When the canonical class of $C$ is very ample, we have that $\Gamma(C,\Omega^1_C)$ is a vector space of dimension $g$, and a choice of basis of holomorphic forms defines an embedding $C \to \P^{g-1}$ called a \textit{canonical embedding}. Any automorphism induces an automorphism of $\Omega^1_C$ by pullback, and therefore extends to an automorphism of $\P^{g-1}$, hence the automorphism groups of canonical curves are all \textit{geometric} in the sense that they act on the embedded curve via projective linear transformations.

It is a classical fact that a general non-hyperelliptic curve of genus $g\ge 3$ admits no nontrivial automorphisms. This motivates the following definition.

\begin{definition} We say a non-hyperelliptic non-singular curve of genus $g$ is \textit{symmetric} if it admits any automorphism other than the identity.
\end{definition}

\subsection{Automorphisms of smooth quartics}\label{subsec:aut-quartics} Here we provide a highly incomplete history of the classification of the automorphism groups of genus three non-singular hyperelliptic curves. This is a history which is scattered throughout the literature, and references papers plagued by minor errors, so we feel it is worth recounting.

In 1878, Schwarz proved that the automorphism group of a compact Riemann surface of genus $g\ge 2$ is finite.  One argues this by observing that for a non-hyperelliptic curve, the automorphism group acts faithfully on the Weierstrass points of the curve, of which there are only finitely many. In 1892, Hurwitz determined an upper bound on the size of such automorphism groups, being $84(g-1)$ \cite{Hurwitz}. This bound is not always attained, however a famous example is the so-called \textit{Klein quartic} (Type I in our classification), which was studied extensively by Felix Klein \cite{Klein}. When $g=3$, the Klein quartic has an automorphism group of size 168, attaining the upper bound proven by Hurwitz.

During this time, Kantor had been working on understanding what kinds of periodic transformations existed on the projective plane, informed by algebraic questions arising from the Hermite and Frobenius programs on quadratic forms \cite{Kantor2}.
In modern language, we might understand this as the study of cyclic subgroups of $\PGL_3(\C)$. In the course of this work, Kantor observed the following result.

\begin{theorem} \cite[p.~40]{Kantor} Given any finite subgroup $G\le\PGL_3(\C)$, there exists a planar curve of genus $g\ge 3$ which is $G$-invariant.
\end{theorem}

Given this result, one might fix a genus $g$ and study which finite groups can act on smooth non-hyperelliptic planar curves of genus $g$, and one may ask what the full automorphism groups of such curves are. Kantor claimed in this same work to classify the automorphism groups of plane quartic curves, however this classification contained many errors which were correctly shortly thereafter by Wiman in 1896 \cite{Wiman}.
In particular, following Hurwitz' work \cite{Hurwitz},
one knows that the prime factors of $\#\Aut(C)$ will only ever be 2, 3, or 7, which was known to Wiman \cite[p.~223]{Wiman}. Although it does not satisfy modern standards of rigor, Wiman correctly identifies all the possible orders of automorphism groups of genus three non-hyperelliptic non-singular curves, and provides an indication as to the generators of these groups.

In 1910, Ciani extended this classification to all irreducible quartics (not necessarily nonsingular) \cite{Ciani-paper}. There are subtleties, however, in such a classification, as we can see in the following example, where we drift beyond the reach of the Hurwitz theorem.

\begin{example} Observe that the quartic $x^3y+z^4$ (which is singular at the point $[0:1:0]$) admits a symmetry of arbitrary order $r\ge 7$ given by $(x,y,z) \mapsto (\zeta_r x, \zeta_r^{r-3}y, z)$ \cite[N.~19]{Ciani-paper}. In particular its automorphism group is necessarily infinite.
\end{example}

Throughout the 20th century, it is unclear which results were known or lost in the literature. Henn recomputed the automorphism groups of genus three non-hyperelliptic curves via the action on Weierstrass points in his 1976 work, obtaining a classification of finite groups occurring as automorphism groups of smooth planar quartics \cite{Henn}. This essentially recovered the computation done by Wiman nearly a century earlier.

Around the mid-1960's, many mathematicians published work studying compact Riemann surfaces equipped with non-trivial involutions. This body of work is far too encompassing to mention all of the relevant players, but we highlight a particular program of mathematics emanating from work of A. Kuribayashi, which studied the moduli of compact Riemann surfaces equipped with non-trivial automorphisms \cite{Kuribayashi66,Kuribayashi76}, leveraging tools from Teichm\"uller theory. In \cite{KK} the authors study smooth projective genus three curves which are cyclic covers of the projective line or of an elliptic curve, and determine their automorphism groups, again recovering the classification initially due to Wiman.

At the dawn of the 21st century, as computer-aided research became increasingly ubiquitous in mathematics, Breuer made tremendous progress by determining all finite groups which act on a curve of genus $\le 48$ \cite{Breuer}. This leveraged group-theoretic constraints on such groups arising from covering space theory, and was accomplished via an exhaustive search through a library of finite groups in GAP \cite{GAP}. Magaard, Shaska, Shpectorov and V\"olkein provide a lovely discussion of how to re-derive the classification for genus $g=3$ from the data obtained by Breuer (that paper also contains rich historical background for the interested reader) \cite{MSSV}.

A beautiful resource for automorphisms of quartics can be found in Dolgachev's book \cite[\S6.5]{Dolgachev}, which recasts the research of Kantor and Wiman in modern language and rigor. Before this book was published, a fantastic paper of Bars \cite{Bars} sketches the classification of automorphism groups of quartics, following both Henn's thesis and the aforementioned paper of Komiya and Kuribayashi. Bars' note corrected errors in the classification appearing in the lecture notes which would later become Dolgachev's book.

With all this in mind, we state the following classification, which we feel it is fair to attribute to Wiman, although we refer the reader to \cite[p.~270]{Dolgachev} for a rigorous proof.

\begin{theorem}\label{thm:aut-gps-quartic} \cite{Wiman} The following groups $G$ are the only groups occurring as the automorphism group of a smooth plane quartic. In each case we provide an equation for quartics $F(x,y,z)$ of that type, with generators for its automorphism group (c.f. \cite{Bars,Dolgachev}). Here $\zeta_n$ denotes a primitive $n$th root of unity, and in $\PSL_2(7)$ we have $\alpha:= \frac{\zeta_7^5 - \zeta_7^2}{\sqrt{-7}}$, $\beta=\frac{\zeta_7^6 - \zeta_7}{\sqrt{-7}}$ and $\gamma = \frac{\zeta_7^3 - \zeta_7^4}{\sqrt{-7}}$
\begin{center}
    \begin{tabular}{|r | c | l | | p{4.7cm} | p{6cm} |}
    \hline
    Type & $\Aut(C)$  & GAP ID  & $F(x,y,z)$ & Generators in $\PGL_3(\C)$\\
    \hline
    \hline
    \hyperref[subsec:Klein]{I} & $\PSL_2(7)$  &  \texttt{[168,42]} & $x^3y+y^3z+z^3x$ & {\genfontsize%
    $\begin{pmatrix}
    \zeta_7^4 & 0 &0 \\
    0 & \zeta_7^2 & 0 \\
    0 & 0 & \zeta_7
    \end{pmatrix}$,
    $\begin{pmatrix}
    0 & 1 & 0 \\
    0 & 0 & 1 \\
    1 & 0 & 0
    \end{pmatrix}$,
    $\begin{pmatrix}
    \alpha &   \gamma & \beta \\
   \beta  &   \alpha & \gamma \\
    \gamma  &   \beta & \alpha     
    \end{pmatrix} $%
    }\\
    \hyperref[subsec:Dyck]{II} & $C_4^{\times 2} \rtimes S_3$  & \texttt{[96,64]} & $x^4 + y^4 + z^4$ & {\genfontsize $\begin{pmatrix} 0 & 0 & 1 \\ 1 & 0 & 0 \\ 0 & 1 & 0 \end{pmatrix}$, $\begin{pmatrix} -i & 0 & 0 \\ 0 & 0 & 1 \\ 0 & i & 0 \end{pmatrix}$}  \\
    \hyperref[subsec:III]{III} & $C_4 \circledcirc A_4$  & \texttt{[48,33]} & $x^4+y^4+z^4 + (4\zeta_3+2)x^2y^2$ & {\genfontsize%
    $\begin{pmatrix} 
    \frac{1+i}{2} & \frac{-1+i}{2} & 0 \\
    \frac{1+i}{2} & \frac{1-i}{2} & 0 \\
    0 & 0 & \zeta_3
    \end{pmatrix}$,
    $\begin{pmatrix} 
    \frac{1+i}{2} & \frac{-1-i}{2} & 0 \\
    \frac{-1+i}{2} & \frac{-1+i}{2} & 0 \\
    0 & 0 & \zeta_3^2
    \end{pmatrix}$} \\
    \hyperref[subsec:IV]{IV} & $S_4$ & \texttt{[24,12]} & $x^4+y^4+z^4+a(x^2y^2+y^2z^2+z^2x^2)$ & {\genfontsize $\begin{pmatrix} 0 & 0 & 1 \\ 1 & 0 & 0 \\ 0 & 1 & 0 \end{pmatrix}$, $\begin{pmatrix} 0 & -1 & 0 \\ 1 & 0 & 0 \\ 0 & 0 & 1 \end{pmatrix}$} \\
    \hyperref[subsec:V]{V} & $P$ & \texttt{[16,13]} & $x^4 + y^4 + z^4 + ax^2y^2$ & {\genfontsize $ \begin{pmatrix} -1 & 0 & 0 \\ 0 & 1 & 0 \\ 0 & 0 & 1 \end{pmatrix}$ ,
    $\begin{pmatrix} i & 0 & 0 \\ 0 & -i & 0 \\ 0 & 0 & 1 \end{pmatrix}$,
    $\begin{pmatrix} 0 & -1 & 0 \\ 1 & 0 & 0 \\ 0 & 0 & 1 \end{pmatrix}$} \\
    \hyperref[subsec:VI]{VI} & $C_9$ & \texttt{[9,1]} & $x^4 + xy^3 + yz^3$ & {\genfontsize $\begin{pmatrix} \zeta_3 & 0 & 0 \\ 0 & 1 & 0 \\ 0 & 0 & \zeta_9 \end{pmatrix}$} \\
    \hyperref[subsec:VII]{VII} & $D_8$ & \texttt{[8,3]} &  $x^4 +y^4 + z^4 + ax^2y^2 + bxyz^2$ & {\genfontsize $\begin{pmatrix} i & 0 & 0 \\
    0 & -i & 0 \\ 0 & 0 & 1\end{pmatrix}$,
    $\begin{pmatrix} 0 & 1 & 0 \\ 1 & 0 & 0 \\ 0 & 0 & 1 \end{pmatrix}$} \\
    \hyperref[subsec:VIII]{VIII} & $C_6$ & \texttt{[6,2]} & $z^3y +x^4 + ax^2y^2 +y^4$ & {\genfontsize $\begin{pmatrix} -1 & 0 & 0 \\ 0 & 1 & 0 \\ 0 & 0 & \zeta_3 \end{pmatrix} $}\\  
    \hyperref[subsec:IX]{IX} & $S_3$  & \texttt{[6,1]} & $x^3z + y^3z + x^2 y^2 + axyz^2+bz^4$ & {\genfontsize$ \begin{pmatrix} \zeta_3 & 0 & 0 \\ 0 & \zeta_3^2 & 0 \\ 0 & 0 & 1 \end{pmatrix}$, $\begin{pmatrix} 0 & 1 & 0 \\ 1& 0 & 0 \\ 0 & 0 & 1 \end{pmatrix}$}\\
    \hyperref[subsec:X]{X} & $K_4$ & \texttt{[4,2]} & $x^4 + y^4 + z^4 + ax^2 y^2 + by^2 z^2 + cx^2 z^2$ & {\genfontsize$\begin{pmatrix} -1 & 0 & 0 \\ 0 & 1 & 0 \\ 0 & 0 & 1 \end{pmatrix}$, $\begin{pmatrix} 1 & 0 & 0 \\ 0 & -1 & 0 \\ 0 & 0 & 1 \end{pmatrix}$}\\    
    \hyperref[subsec:XI]{XI} & $C_3$ & \texttt{[3,1]} & $z^3y + x(x-y)(x-ay)(x-by)$ 
    & {\genfontsize$\begin{pmatrix} 
    1 & 0 & 0 \\
    0 & 1 & 0 \\
    0 & 0 & \zeta_3
    \end{pmatrix}$} \\
    \hyperref[subsec:XII]{XII} & $C_2$ & \texttt{[2,1]} & $x^4 + y^4 + z^4 + x^2(ay^2 + byz + cz^2) + dy^2 z^2$ & {\genfontsize$ \begin{pmatrix} -1 & 0 & 0 \\ 0 & 1 & 0 \\ 0 & 0 & 1 \end{pmatrix}$ }\\
    \hline
    \end{tabular}
\end{center}
The equations are subject to the following constraints:
\begin{itemize}
    \item Type IV: we have $a \ne 0$ since this yields a Type II curve, and $a\ne \frac{3}{2}(-1 \pm \sqrt{-7})$, otherwise one obtains a Type I curve.
    \item Type V: $a \notin \left\{ 0,\ \pm 2\sqrt{-3},\ \pm 6 \right\}$. If $a=0$ or $a= \pm6$ we obtain a Type II quartic, and if $a = \pm 2\sqrt{-3}$ we obtain a Type III quartic.
    \item Type VII: $b\ne 0$. For $b=0$ we obtain a Type V quartic.
    \item Type VIII: $a\ne 0$. For $a=0$ we obtain a Type III quartic.
    \item Type IX:\footnote{%
    Our equation differs slightly from \cite{Dolgachev} by rescaling variables.
    }%
    $a\ne 0$, otherwise we obtain a Type IV quartic.
    \item Type X: all the values $\left\{ \pm a, \pm b, \pm c \right\}$ are distinct, otherwise we obtain a Type VII curve.
    \item Type XI: if either $a=\zeta_3$ or $b=\zeta_3$ we obtain a Type VI curve. If either $a=-1$ or $b=-1$ we obtain a Type VIII curve. 
    \item Type XII:\footnote{%
    Our equation for a Type XII quartic matches that found in \cite[Theorem~6.5.2]{Dolgachev}, which is a priori different from \cite[Lemma~6.5.1]{Dolgachev} and \cite{Henn,Bars}. A Type XII quartic is of the form $x^4 + x^2L_2(y,z) + L_4(y,z)$, where $L_k$ denotes a homogeneous binary form of degree $k$. By the classical theory of quantics, we recall that every binary quartic form admits a \textit{canonical form} $y^4 + dy^2 z^2 + z^4$ after a projective linear transformation, a fact that dates back to Cayley but can be found in \cite[Chapter~XII]{Elliott}. This is the source of the quartic equation we use here.%
    }
    If $b=0$ we obtain a Type X quartic. If $a=-2$, $b=0$, and $c=-d$, we obtain a Type IX curve.

\end{itemize}
\end{theorem}

The reader may notice subgroup containment between some of these groups --- we discuss how this manifests in our computations in \autoref{sec:subgroups}.

\begin{remark} The quartic equations appearing in \autoref{thm:aut-gps-quartic} are in fact the only quartic equations (up to projective transformation) with the given automorphism group. One proves this by first providing normal forms for cyclic actions, then doing a case-by-case analysis as in \cite[\S6.4]{Dolgachev}. When free parameters appear, they correspond to the dimension of the moduli of smooth planar quartics having the given automorphism group. In that sense, the free coefficients in the equations \autoref{thm:aut-gps-quartic} give an indication of how the moduli of symmetric quartics is stratified.
\end{remark}

\section{The Euler number for bitangents} \label{section:euler} 

The proof of Theorem \ref{thm:main} relies on the fact that we can obtain a $G$-equivariant count of the 28 bitangents to a smooth quartic $C$ for $G\cong\text{Aut}(C)$ using the equivariant Euler number of \cite{EEG}, and this count will be independent of choice of quartic with automorphism group $G$. We use the bundle parameterizing bitangents to smooth quartics following the construction of \cite{LarsonVogt}, which we recount exactly below. 

Let
\begin{align*}
    \mathcal{S} \to (\P^2)^\vee
\end{align*}
denote the tautological bundle over the moduli of lines in $\mathbb{P}^2$. Let $X:= \mathbb{P}(\Sym^2 \mathcal{S}^\vee)$ so that a point on $X$ is a pair $(L,Z)$ where $L \subseteq \P^2$ is a line and $Z \subseteq L$ is a degree two subscheme.  Writing $\pi\colon X\to \mathbb{P}^{2\vee}$ for the natural projection map, we have the commutative diagram:
\begin{equation} \label{diag:pullback_defining_X}
 \begin{tikzcd}
    \pi^\ast\Sym^2 \mathcal{S}^\vee\rar\dar & \Sym^2 \mathcal{S}^\vee\dar\\
    X\rar["\pi" below] & (\P^2)^\vee.
\end{tikzcd} 
\end{equation}
Note that any fiber of $\mathcal{O}_X(-1)\subseteq \pi^*\Sym^2\mathcal{S}^\vee$ determines a choice of equations defining $Z\subseteq L$, whence the quotient of
\begin{equation}
    \mathcal{O}_X(-1) \to \pi^\ast \Sym^2 \mathcal{S}^\vee
\end{equation}
parameterizes quadratic equations on $L$ vanishing along $Z$. 
Squaring, we obtain the composite
\begin{equation}\label{eqn:sym2-sym4}
\begin{aligned}
     \mathcal{O}_X(-2) \to \Sym^2 \mathcal{S}^\vee \otimes \Sym^2 \mathcal{S}^\vee \to  \Sym^4 \mathcal{S}^\vee.
\end{aligned}
\end{equation}
The vector bundle for our enumerative problem will be the quotient bundle $$E:=\Sym^4\mathcal{S}^\vee/\mathcal{O}_X(-2),$$ which will parameterize quartic polynomials on $L$ vanishing at the square of an equation of $Z$. In other words, the quotient $\Sym^4\mathcal{S}^\vee\to E$ over the fiber $(L,Z)$ in 
\begin{equation}\label{eq:SES_defining_E} 0\to \mathcal{O}_X(-2)\to \Sym^4\mathcal{S}^\vee\to E\to 0 
\end{equation}
realizes $E$ as the evaluation of a quartic polynomial on $L$  along $2Z$. 

\begin{theorem}\label{thm:euler-class} Let $G\le \PGL_3$ be contained in any group appearing as the automorphism group of a smooth non-hyperelliptic plane quartic. Then $$E\to X$$ constructed above is a $G$-equivariant vector bundle, and any planar quartic $C$ with $G\le \Aut(C)$ induces a $G$-equivariant section of this bundle whose zeros count the bitangents to $C$.
\end{theorem}

\begin{proof}
The $G$-action on $\mathbb{P}^2$ defines an action on $\mathbb{P}^{2\vee}$ and thus on $\mathcal{S}\to(\mathbb{P}^2)^\vee$. The symmetric power $\Sym^2\mathcal{S}^\vee\to (\mathbb{P}^{2})^{\vee}$ is thus a $G$-vector bundle, and the projectivization $X:=\mathbb{P}(\Sym^2\mathcal{S}^\vee) \stackrel{\pi}{\to}(\mathbb{P}^{2})^{\vee}$ is as well. As the actions on $X$, $\Sym^2\mathcal{S}^\vee$, and $\pi^*\Sym^2\mathcal{S}^\vee$ are all inherited from that on $(\mathbb{P}^2)^\vee$, diagram \eqref{diag:pullback_defining_X} commutes equivariantly. 

Since the $G$-action on $\mathcal{O}_X(-1)$ is inherited from that of $\pi^*\Sym^2\mathcal{S}^\vee$, the inclusion $\mathcal{O}_X(-1) \to \pi^\ast \Sym^2 \mathcal{S}^\vee$ is $G$-equivariant. Squaring, we deduce that\[
\mathcal{O}_X(-2)\hookrightarrow(\pi^*\Sym^2\mathcal{S}^\vee)^{\otimes 2}
\]
is $G$-equivariant as well. Since over any particular $(L,Z)$ in $X$, the composition 
$$\mathcal{O}_X(-2)_{(L,Z)}\hookrightarrow((\Sym^2\mathcal{S}^\vee)^{\otimes 2})_{(L,Z)} \to (\Sym^4\mathcal{S}^\vee)_{(L,Z)}$$ 
is just the composition of inclusions of the space spanned by the square of a quadratic on $Z$, $\langle q_Z^2\rangle\mapsto \langle q_Z^2\rangle\mapsto \langle q_Z^2\rangle$, into the space $(\Sym^4\mathcal{S}^\vee)_{(L,Z)}$ of all degree 4 polynomials on $L$, 
$$\mathcal{O}_X(-2)\to \Sym^4\mathcal{S}^\vee$$ 
is $G$-equivariant. In order to argue that the quotient bundle $E$ can be chosen equivariantly, we pick a $G$-invariant Hermitian metric in $\Sym^4 \mathcal{S}^\vee$ in order to induce an equivariant splitting of the short exact sequence of bundles
\begin{align*}
    \mathcal{O}_X(-2) \to \Sym^4 \mathcal{S}^\vee \to E.
\end{align*}
Therefore $E \to X$ is a $G$-equivariant vector bundle, and $G$-equivariant planar quartics induce equivariant sections of the bundle.
\end{proof}

We state this conclusion again explicitly: 

\begin{corollary}\label{cor:bitangents_give_euler_class} Let $G$ be as in \autoref{thm:aut-gps-quartic}. Then the bitangents of any smooth planar quartic which is $G$-symmetric  have $G$-orbits given by the Euler class $n_G(E)$.
\end{corollary}

If $f$ is a quartic polynomial defining a $G$-symmetric, smooth, non-hyperelliptic plane quartic $C$ and $\sigma_f$ is the section $X\to E$ determined by $C$, then $\sigma_f$ vanishing at $(L,Z)$ exactly means that $L$ is a bitangent of $C$ with points of bitangency at $Z\subset L$. Thus the fact that $\sigma_f$ is a $G$-equivariant section and $C$ has 28 bitangents exactly means that $Z(\sigma_f)=\{(L,Z)\colon \sigma_f(L,Z)=0\}$ is a size 28 $G$-set, and $G$ acts transitively on the disjoint union of orbits of bitangents and points of bitangency of $C$. This count of bitangents is independent of symmetric quartic $C$ with $\Aut(C)\cong G$ since the Euler number is independent of choice of section:

\begin{corollary}\label{cor:independence} Let $G$ be any group which appears as the automorphism group of a genus three non-hyperelliptic curve. Then for any such curve $C$ with $G \cong\Aut(C)$, the action of $G$ on the bitangents of $C$ is independent of the choice of $C$.
\end{corollary}
\begin{proof} This is a direct consequence of equivariant conservation of number, pulling back an Euler class for equivariant homotopical bordism to the Burnside ring of the group \cite[Theorem~5.24]{EEG}.
\end{proof}

Combining Theorem \ref{thm:euler-class} and Corollaries \ref{cor:bitangents_give_euler_class} and \ref{cor:independence} proves \autoref{thm:main}. Formulas in $A(G)$ for the 28 bitangents of $G$-symmetric quartics are given in the table included with the statement of the main theorem in the introduction (see Theorem \ref{thm:main}), which we omit here for brevity. 

\begin{remark}\label{rmk:restriction}
It is worth noting that for any subgroup $H$ of $G$ 
we have a restriction map $\Res^G_H\colon A(G)\to A(H)$ that sends the class of any $G$-set to the $H$-set with $H$-action induced from the action of $G$. In particular, any $G$-symmetric quartic $C$ is $H$-symmetric, and a $G$-action on the bitangents of $C$ descends to an $H$-action on the bitangents of $C$ via restriction. We explore consequences of this for equivariant counts of bitangents in \autoref{sec:subgroups}.
\end{remark}

\section{Symmetric quartics and their bitangents}\label{section:computations}

In this section we compute the action of the automorphism group of a symmetric quartic on its 28 bitangents for all 12 types. For $H \le G$ a subgroup, we denote by $[G/H]$ the isomorphism class of the $G$-set of left cosets of $H$ in $G$. Every genuine, finite $G$-set decomposes as such a sum.

\subsection{Type I: Klein quartics}\label{subsec:Klein}
The most well-studied symmetric quartic, which boasts its own book \cite{Levy}, and its own sculpture at SLMath (formerly MSRI)\footnote{The sculpture is called \textit{The Eightfold Way}, created by Helaman Ferguson and installed in 1993. Thurston wrote a paper about the Klein quartic in honor of the statue's inauguration \cite{Thurston}.}, is the so-called \textit{Klein quartic}:
\begin{align*}
    x^3 y + y^3 z + z^3 x.
\end{align*}
Its automorphism group $G_{168}$ is the maximum possible by the Hurwitz bound, making it a prototypical example of what is called a \textit{Hurwitz curve}. Its automorphism group is isomorphic to the projective special linear group over the finite field with seven elements, which acts on $\P^2$ in the following way:
\begin{align*}
    \PSL_2(7) \cong \left\langle \begin{pmatrix} 0 & 0 & 1 \\ 1 & 0 & 0 \\ 0 & 1 & 0 \end{pmatrix}, \begin{pmatrix} -i & 0 & 0 \\ 0 & 0 & 1 \\ 0 & i & 0 \end{pmatrix} \right\rangle\le \PGL_3(\C). 
\end{align*}
We give a partial indication as to how to visualize this group. The Klein quartic, considered as a Riemann surface, can be tiled with 24 heptagons, and in fact it is a quotient of the hyperbolic plane via the congruence subgoup $\Gamma(7) \nsubgp \PSL(2,\Z)$, whose quotient $\PSL_2(7)$ acts via deck transformations and becomes the automorphism group of the Klein quartic. This group can be understood as the automorphisms attached to the tiling --- a heptagon can be sent to any other of the 24 heptagons, and has a sevenfold rotation, for a grand total of $24 \times 7 = 168 = \#\PSL_2(7)$ symmetries.\footnote{Reflections of the heptagons are not considered since they don't preserve the complex structure.}

Alternatively via the automorphism $\PSL_2(7) \cong \PSL_3(2)$, the automorphism group of the Klein quartic can be visualized as the symmetries of the \textit{Fano plane} as in \autoref{fig:Fano-plane}.

\begin{figure}[h]
  \includegraphics[width=0.2\linewidth]{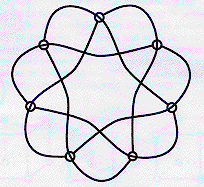}
  \centering
  \caption{An artistic rendering of the sevenfold symmetries on the Fano plane, drawn by Burkard Polster, found on John Baez' blog \cite{Baez}.}
  \label{fig:Fano-plane}
\end{figure}

The automorphism group acts on bitangents, so we recover our first computation, which was undoubtedly known to Klein.

\begin{computation}\label{comp:I} The automorphism group of the Klein quartic acts transitively on the 28 bitangents, with isotropy given by the symmetric group $S_3 \le \PSL_2(7)$.
\end{computation}

The Klein has exactly four real bitangents, which we visualize in \autoref{fig:Klein}.

\begin{figure}[h]
  \includegraphics[width=\quarticpicwidth]{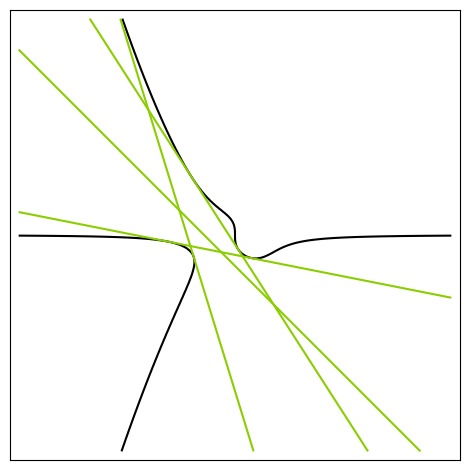}
  \centering
  \caption{The four real bitangents graphed on the Klein quartic. All lines lie in the same orbit.}
  \label{fig:Klein}
\end{figure}

\subsection{Type II: the Dyck quartic}\label{subsec:Dyck}
The quartic of type two has equation
\begin{align*}
    x^4 + y^4 + z^4,
\end{align*}
and was first studied in an 1880 paper by W. Dyck \cite{Dyck}, who ascertained that it admits 96 symmetries. The automorphism group has isomorphism type $G_{96}\cong C_4^{\times 2 }\rtimes S_3$, acting on the projective plane in the following way:
\begin{equation}\label{eqn:type-II-grp}
\begin{aligned}
     C_4^{\times 2} \rtimes S_3 \cong \left\langle \begin{pmatrix} 0 & 0 & 1 \\ 1 & 0 & 0 \\ 0 & 1 & 0 \end{pmatrix}, \begin{pmatrix} -i & 0 & 0 \\ 0 & 0 & 1 \\ 0 & i & 0 \end{pmatrix} \right\rangle\le \PGL_3(\C).
\end{aligned}
\end{equation}
\begin{computation}\label{comp:II} Any Type II quartic has bitangents with orbits
\begin{align*}
    [C_4^{\times 2} \rtimes S_3/C_6] + [C_4^{\times 2} \rtimes S_3/C_8].
\end{align*}
\end{computation}
While the Dyck quartic has no real points, it does admit real bitangents, where the points of tangency are complex. These are given by
\begin{align*}
&x - y - z \\
&x + y - z \\ 
&x - y + z \\
&x + y + z, 
\end{align*}
and they all lie in the $[C_4^{\times 2} \rtimes S_3/C_6]$ orbit.

\begin{remark}\label{rmk:edge}
Edge wrote an entire paper about this quartic in 1938 \cite{Edge}, in which he investigates its determinantal forms. Explicit equations for each of the 28 bitangents is included in \cite[p.~14]{Edge}, though he attributes this computation to Cayley. In \cite[\S14]{Edge}, Edge extends his computations of bitangents and Cayley octads for the Dyck quartic to a more general family of quartics, mainly those of the form
\begin{align*}
    (1+\gamma^2)^2 (x^4 + y^4 + z^4) - 2 (1+\gamma^4) (y^2 z^2 + x^2 y^2 + x^2 z^2)  = 0.
\end{align*}
When $\gamma = \zeta_8$ is a primitive $8$th root of unity, we obtain the Dyck quartic. For other values of $\gamma$, quartics of this shape are called \textit{Edge quartics}. The Edge quartic with $\gamma = \pm\sqrt{5}$, from \autoref{exa:edge-D8}, was studied in \cite[p.~713]{PSV}; we note via our classification the automorphism group is different --- this Edge quartic is a Type IV quartic with $a = - \frac{34}{25}$.
\end{remark}

\subsection{Type III: Order 48}\label{subsec:III}
Order 48, structure $C_4\circledcirc A_4$.

There is a unique Type III quartic given by the equation:
\begin{align*}
    x^4+y^4+z^4 + (4\zeta_3+2)x^2y^2,
\end{align*}
with automorphism group given by
\begin{align*}
    C_4 \circledcirc A_4 &= \left\langle
    \begin{pmatrix} 
    \frac{1+i}{2} & \frac{-1+i}{2} & 0 \\
    \frac{1+i}{2} & \frac{1-i}{2} & 0 \\
    0 & 0 & \zeta_3
    \end{pmatrix},
    \begin{pmatrix} 
    \frac{1+i}{2} & \frac{-1-i}{2} & 0 \\
    \frac{-1+i}{2} & \frac{-1+i}{2} & 0 \\
    0 & 0 & \zeta_3^2
    \end{pmatrix}\right\rangle \le \PGL_3(\C).
\end{align*}

A GAP computation shows that the 28 bitangents break into two orbits of sizes 24 and four. The first orbit has isotropy $C_2 \le C_4 \circledcirc A_4$, which is the non-normal cyclic subgroup of order two in $A_4$, while the second orbit has isotropy given by the unique subgroup of order 12, which is cyclic:

\begin{computation} The bitangents to a Type III quartic have orbits
\begin{align*}
    [C_4 \circledcirc A_4 / C_2] + [C_4 \circledcirc A_4 / C_{12}].
\end{align*}
\end{computation}
 
There are no real points, so we unfortunately cannot visualize this example.

\subsection{Type IV: The pencil of octahedral quartics}\label{subsec:IV}
Order 24, structure $S_4$

Up to change of coordinates, Type IV quartics are given by the following pencil, which we call the pencil of \textit{octahedral quartics}
\begin{align*}
    x^4 + y^4 + z^4 + a (x^2 y^2 + y^2 z^2 + z^2 x^2),
\end{align*}
for $a\ne \frac{3}{2}(-1\pm \sqrt{-7})$ (this would give the Klein) and $a\ne 0$ (this would give the Dyck quartic). As $a$ varies, we sweep out a pencil of quartics, first studied by Wiman and Ciani independently. A modern discussion of this pencil of quartics can be found in \cite[\S8]{Dolgachev-Kanev}. The paper \cite{Rodriguez-G-A} is dedicated solely to the study of this pencil, where the authors study (among other things) the Fuschian uniformizations and the Jacobians of these quartics.

A Type IV quartic has automorphism group isomorphic to the symmetric group $S_4$. For any such $a$ as above, we see the automorphism group is
\begin{align*}
    S_4 \cong \left\langle \begin{pmatrix} 0 & 0 & 1 \\ 1 & 0 & 0 \\ 0 & 1 & 0 \end{pmatrix}, \begin{pmatrix} 0 & -1 & 0 \\ 1 & 0 & 0 \\ 0 & 0 & 1 \end{pmatrix} \right\rangle \le \PGL_3(\C).
\end{align*}
\begin{computation}\label{comp:IV} The bitangents of a Type IV quartic have orbits
\begin{align*}
    [S_4/C_2^o] + [S_4/C_2^e] + [S_4/S_3].
\end{align*}
Here $C_2^o$ is an \textit{odd} copy of $C_2\le S_4$, i.e. it is a single transposition, while $C_2^e$ is a product of two disjoint transpositions, hence it is \textit{even}.
\end{computation}

\begin{example}
When $a=-3$, we compute there are 16 real bitangents, forming the orbits $[S_4/C_2^o]$ and $[S_4/S_3]$. We visualize this in \autoref{fig:IV}.
\begin{figure}[H]
  \includegraphics[width=\quarticpicwidth]{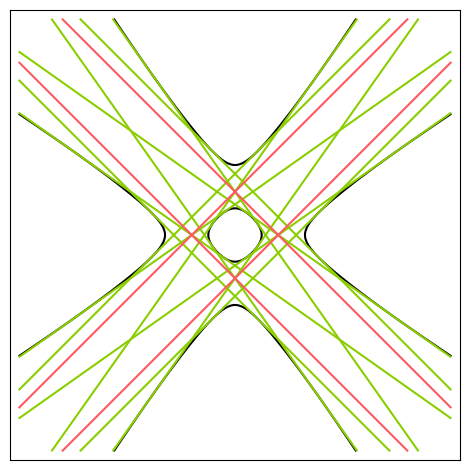}
  \centering
  \caption{The orbits of the real bitangents on the Type IV quartic with $a=-3$.
  Lines in \linecolorone~are in the $[S_4/C_2^o]$ orbit, and lines in \linecolortwo~are in the $[S_4/S_3]$ orbit.}
  \label{fig:IV}
\end{figure}
\end{example}

\subsection{Type V: Order 16}\label{subsec:V}
Order 16, structure $P = C_4 \circledcirc K_4$ the \textit{Pauli group}.

\begin{notation} The automorphism group of a type V surface is the \textit{Pauli group}, which we denote by $P$. In particular there are three non-central conjugacy classes of cyclic subgroups of order two, which we denote by $C_2^{(1)}$, $C_2^{(2)}$, and $C_2^{(3)}$, and the center is a cyclic group of order four which we denote by $C_4^Z$.
\end{notation}

Type V quartics are given by the equation
\begin{align*}
    x^4 + y^4 + z^4 + ax^2 y^2
\end{align*}
for $a\notin \left\{ 0,\ \pm 2\sqrt{-3},\ \pm 6 \right\}$. Their automorphism groups are of the form
\begin{align*}
    P \cong \left\langle \begin{pmatrix} -1 & 0 & 0 \\ 0 & 1 & 0 \\ 0 & 0 & 1 \end{pmatrix} ,
    \begin{pmatrix} i & 0 & 0 \\ 0 & -i & 0 \\ 0 & 0 & 1 \end{pmatrix},
    \begin{pmatrix} 0 & -1 & 0 \\ 1 & 0 & 0 \\ 0 & 0 & 1 \end{pmatrix}
    \right\rangle\le \PGL_3(\C).
\end{align*}

\begin{computation}\label{comp:V} Any quartic of Type V has bitangents with orbits
\begin{align*}
    [P/C_2^{(1)}] + [P/C_2^{(2)}] + [P/C_2^{(3)}] + [P/C_4^Z].
\end{align*}
\end{computation}

\begin{example} When $a=-4$ we see eight real bitangents, as in \autoref{fig:V}.
\begin{figure}[h]
  \includegraphics[width=\quarticpicwidth]{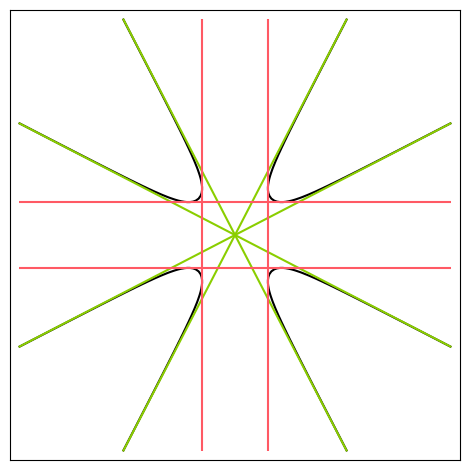}
  \centering
  \caption{The orbits of the real bitangents on a Type V quartic. The four in \linecolorone~form the the orbit whose isotropy is the central $C_4^Z$. The four lines in \linecolortwo~are part of an orbit with $C_2^{(1)}$ isotropy.}
  \label{fig:V}
\end{figure}
\end{example}

\subsection{Type VI: Order 9}\label{subsec:VI}
Order 9, structure $C_9$.

There is a unique Type VI quartic, given by the equation
\begin{align*}
    x^4 + xy^3 + yz^3,
\end{align*}
whose automorphism group is cyclic of order nine:
\begin{align*}
    C_9 \cong \left\langle \begin{pmatrix} \zeta_3 & 0 & 0 \\ 0 & 1 & 0 \\ 0 & 0 & \zeta_9 \end{pmatrix}  \right\rangle \le \PGL_3(\C).
\end{align*}
\begin{computation}\label{comp:VI} The Type VI quartic has bitangents with orbits
\begin{align*}
    3[C_9/e] + [C_9/C_9].
\end{align*}
\end{computation}
This quartic has four real lines, one in each orbit, in \autoref{fig:VI}.
\begin{figure}[h]
  \includegraphics[width=\quarticpicwidth]{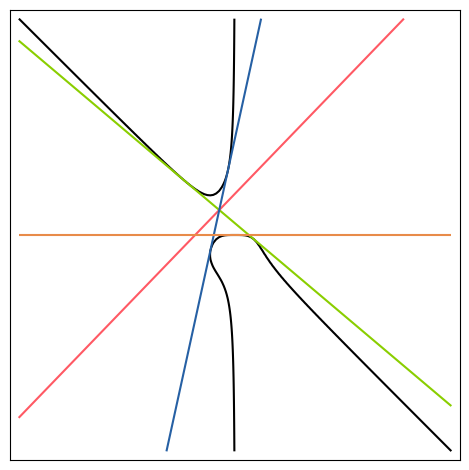}
  \centering
  \caption{The orbits of real bitangents on the quartic of Type VI. In \linecolorone~are lines in the first trivial isotropy orbit.
In \linecolortwo~are lines in the 2nd trivial isotropy orbit.
In \linecolorthree~are lines in the 3rd trivial isotropy orbit.
Finally in \linecolorfour~are lines in the $[C_9/C_9]$ orbit.}
  \label{fig:VI}
\end{figure}

\subsection{Type VII: Order 8}\label{subsec:VII}
Order 8, structure $D_8$.

Type VII quartics are given by the equation
\begin{align*}
    x^4 + y^4 + z^4 + ax^2 y^2 + bxyz^2
\end{align*}
for $b\ne 0$. The automorphism group is given by
\begin{align*}
    D_8 \cong \left\langle \begin{pmatrix} i & 0 & 0 \\
    0 & -i & 0 \\ 0 & 0 & 1\end{pmatrix},
    \begin{pmatrix} 0 & 1 & 0 \\ 1 & 0 & 0 \\ 0 & 0 & 1 \end{pmatrix} 
    \right\rangle\le \PGL_3(\C).
\end{align*}
Writing the first generator as $r$ and the second as $s$, we have three conjugacy classes of subgroups of order two, namely
\begin{align*}
    C_2^{(1)} &:= \left\langle s \right\rangle \cong \left\langle r^2 s \right\rangle \\
    C_2^{(2)} &:= \left\langle rs \right\rangle \cong \left\langle r^3 s \right\rangle \\
    C_2^Z &:= \left\langle r^2 \right\rangle,
\end{align*}
where the superscript $Z$ denotes that it is central.

\begin{computation}\label{comp:D8}
The bitangents of any Type VII quartic have orbits
\begin{align*}
    [D_8/e] + 2[D_8/C_2^{(1)}] + 2[D_8/C_2^{(2)}] + [D_8/C_2^Z].
\end{align*}
\end{computation}

\begin{example} For $a=-3$ and $b=1$, we see eight real bitangents, pictured in \autoref{fig:VII}.
\begin{figure}[h]
  \includegraphics[width=\quarticpicwidth]{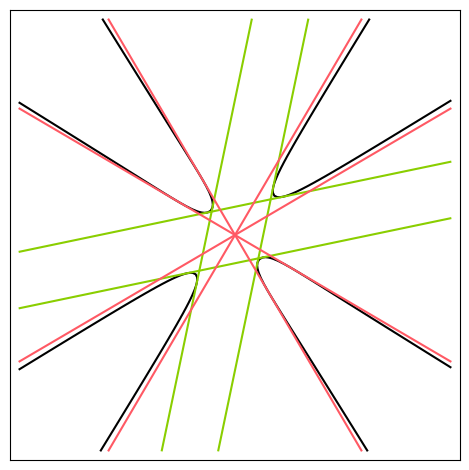}
  \centering
  \caption{The orbits of real bitangents on an example quartic of Type VII. The \linecolorone~lines have trivial isotropy, i.e. they live in the $[D_8/e]$ orbit, while the \linecolortwo~-lines form the $[D_8/C_2^Z]$ orbit.}
  \label{fig:VII}
\end{figure}
\end{example}

\subsection{Type VIII: Order 6}\label{subsec:VIII}
Order 6, structure $C_6$.

Type VIII quartics are given by
\begin{align*}
    x^4 + y^4 + ax^2 y^2 + yz^3
\end{align*}
for $a\ne 0$. Its automorphism group is given by
\begin{align*}
    C_6 \cong \left\langle \begin{pmatrix} -1 & 0 & 0 \\
    0 & 1 & 0 \\
    0 & 0 & \zeta_3\end{pmatrix}\right\rangle\le \PGL_3(\C).
\end{align*}
\begin{computation}\label{comp:VIII} The bitangents of any Type VIII quartic have orbits
\begin{align*}
    4[C_6/e] + [C_6/C_2] + [C_6/C_6].
\end{align*}
\end{computation}

\begin{example} When $a=-3$, we see four real bitangents, two occurring in distinct orbits with trivial isotropy, one in the orbit with $C_3$ isotropy, and one fixed under the action of the automorphism group, pictured in \autoref{fig:VIII}.
\begin{figure}[h]
  \includegraphics[width=\quarticpicwidth]{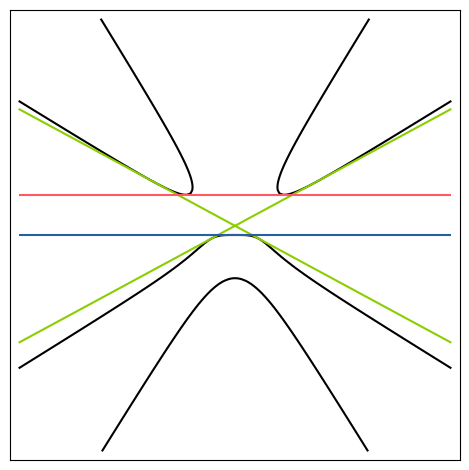}
  \centering
  \caption{The orbits of real bitangents on an example quartic of Type VIII. In \linecolorone~are lines in the $[C_6/e]$ orbit.
In \linecolortwo~are lines in the $[C_6/C_2]$ orbit.
In \linecolorthree~are lines in the $[C_6/C_6]$ orbit.}
  \label{fig:VIII}
\end{figure}
\end{example}

\subsection{Type IX: Order 6}\label{subsec:IX}
Order 6, structure $S_3$

Type IX quartics are given by the equation
\begin{align*}
    x^3z + y^3z + x^2 y^2 + axyz^2 + bz^4
\end{align*}
for $a\ne 0$. 
Its automorphism group is given by
\begin{align*}
    S_3 \cong \left\langle \begin{pmatrix} \zeta_3 & 0 & 0\\
    0 & \zeta_3^2 & 0 \\
    0 & 0 & 1\end{pmatrix} , \begin{pmatrix} 0 & 1 & 0 \\
    1 & 0 & 0 \\
    0 & 0 & 1\end{pmatrix} \right\rangle\le \PGL_3(\C).
\end{align*}
\begin{computation}\label{comp:IX} The bitangents on any Type IX quartic are given by
\begin{align*}
     3[S_3/e] + 3[S_3/C_2] + [S_3/S_3].
\end{align*}
\end{computation}

\begin{example} When $a=-25$ and $b=10$, we compute eight real bitangents, pictured in \autoref{fig:IX}.
\begin{figure}[h]
  \includegraphics[width=\quarticpicwidth]{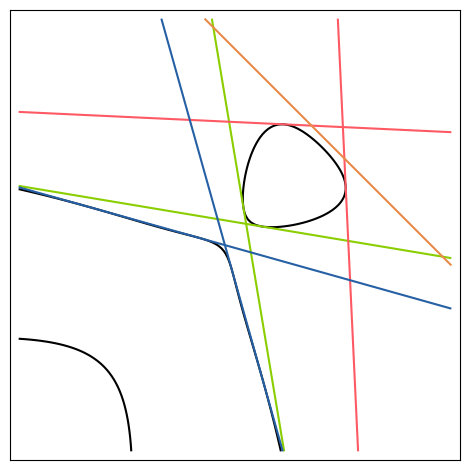}
  \centering
  \caption{The orbits of real bitangents on an example quartic of Type IX. In \linecolorone,\linecolortwo,~and \linecolorthree~are lines in different $[S_3/e]$ orbits. The line in \linecolorfour~lies in an $[S_3/C_2]$ orbit. Not pictured is a bitangent $z=0$ at infinity which forms the $[S_3/S_3]$ orbit.}
  \label{fig:IX}
\end{figure}
\end{example}

\subsection{Type X: Order 4}\label{subsec:X}
Order 4, structure $K_4$

Type X quartics are given by the equation
\begin{align*}
    x^4 + y^4 + z^4 + ax^2 y^2 + by^2 z^2 + cx^2 z^2
\end{align*}
for $a,b,c$ all distinct. Its automorphism group is given by
\begin{align*}
    K_4 \cong \left\langle
    \begin{pmatrix}
    -1 & 0 & 0 \\
    0 & 1 & 0 \\
    0 & 0 & 1
    \end{pmatrix},
    \begin{pmatrix}
    1 & 0 & 0 \\
    0 & -1 & 0\\
    0 & 0 & 1
    \end{pmatrix} 
    \right\rangle \le \PGL_3(\C).
\end{align*}
\begin{computation}\label{comp:X} The bitangents on any Type X quartic have orbits
\begin{align*}
    5[K_4/e] + 2[K_4/C_2^L] + 2[K_4/C_2^R].
\end{align*}
\end{computation}

\begin{example} When $(a,b,c) = (-9,-3,-8)$, we compute 16 real bitangents, forming four orbits with trivial isotropy, pictured in \autoref{fig:X}.
\begin{figure}[h]
  \includegraphics[width=\quarticpicwidth]{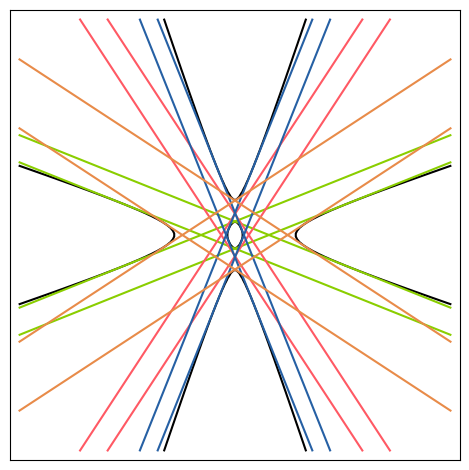}
  \centering
  \caption{The orbits of real bitangents on an example quartic of Type X. There are 16 bitangents, forming four orbits of the form $[K_4/e]$.}
  \label{fig:X}
\end{figure}
\end{example}

\subsection{Type XI: Order 3}\label{subsec:XI}
Order 3, structure $C_3$

Type XI quartics are given by the equation
\begin{align*}
    x(x-y)(x-ay)(x-by) + yz^3,
\end{align*}
for $a,b$ such that $a\neq 1$, $b\neq 1-a$, and $(x-a)(x-b)\neq x^2+x+1$. It has automorphism group
\begin{align*}
    C_3 \cong \left\langle \begin{pmatrix} 
    1 & 0 & 0 \\
    0 & 1 & 0 \\
    0 & 0 & \zeta_3
    \end{pmatrix}  \right\rangle\le \PGL_3(\C).
\end{align*}
\begin{computation}\label{comp:XI} The bitangents on any Type XI quartic have orbits
\begin{align*}
    9[C_3/e] + [C_3/C_3].
\end{align*}
\end{computation}

\begin{example} When $a=2$ and $b=3$, we compute four real bitangents, pictured in \autoref{fig:XI}.
\begin{figure}[h]
  \includegraphics[width=\quarticpicwidth]{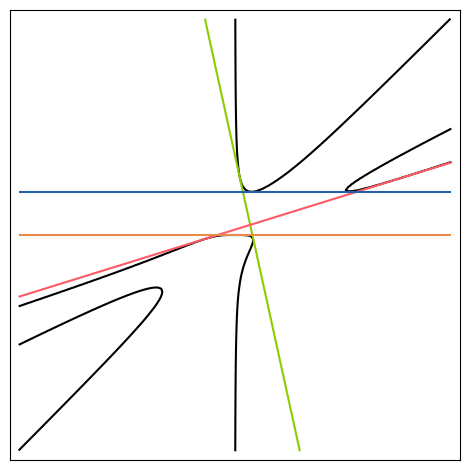}
  \centering
  \caption{The orbits of real bitangents on an example quartic of Type XI. Orbits with colors \linecolorone,~\linecolortwo,~and \linecolorthree~have trivial isotropy. The \linecolorfour~line is in a $[C_3/C_3]$ orbit, i.e. it is fixed by the automorphism group.}
  \label{fig:XI}
\end{figure}

\end{example}

\subsection{Type XII: Order 2}\label{subsec:XII}
Order 2, structure $C_2$.

All Type XII quartics are given by
\begin{align*}
    x^4 + y^4 + z^4 + x^2(ay^2 + byz + cz^2) + dy^2 z^2,
\end{align*}
for $b\ne 0$, $a\ne -2$, and $c\neq -d$. Note that if $b=0$ we obtain a Type X quartic, and if $a=-2$, $b=0$, and $c=-d$ we obtain a Type IX curve.

This has automorphism group
\begin{align*}
    C_2 \cong \left\langle 
    \begin{pmatrix} 
    -1 & 0 & 0 \\
    0 & 1 & 0 \\
    0 & 0 & 1\end{pmatrix} \right\rangle\le \PGL_3(\C).
\end{align*}
\begin{computation}\label{comp:XII} The bitangents on any Type XII quartic have orbits
\begin{align*}
    12[C_2/e] + 4[C_2/C_2].
\end{align*}
\end{computation}

\begin{example}
For the Type XII quartic defined by the equation
\begin{align*}
    x^4 + x^2 (-y^2 - 4z^2) + y^4 - y^2 z^2 - z^4,
\end{align*}
we have four real bitangents, forming three orbits, pictured in \autoref{fig:XII}.
\begin{figure}[h]
  \includegraphics[width=\quarticpicwidth]{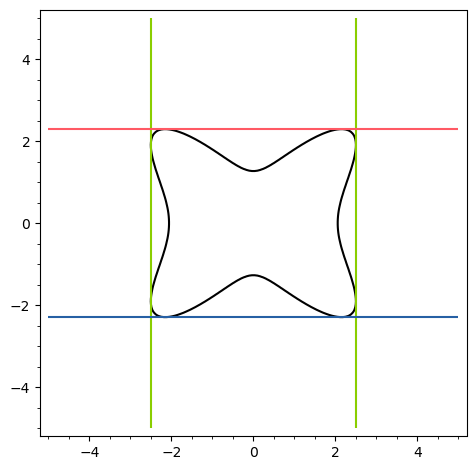}
  \centering
  \caption{The orbits of real bitangents on an example quartic of Type XII. Lines in \linecolorone~have trivial isotropy, while the lines in \linecolortwo~and \linecolorthree~are fixed by the automorphism group.}
  \label{fig:XII}
\end{figure}
\end{example}

\section{The subgroup lattice and restriction}\label{sec:subgroups}

As mentioned in \autoref{rmk:restriction}, we may restrict our attention to a subgroup $H \le G \cong \Aut(C)$ of automorphisms for a given curve, and the action of that group on the bitangents is the restriction of the $G$-set of bitangents to the subgroup $H$. As we have seen in \autoref{thm:aut-gps-quartic}, we have families of curves attached to certain values of $a$, $b$, etc., and when these take particular values the type of the quartic can change. Often how this occurs is that the entire family is invariant under the automorphism group, however for a special value of a coefficient, a new polynomial relation holds, introducing an extra automorphism to the equation and enlarging the automorphism group of the curve. To that end, it is important to understand which automorphism groups are subgroups of one another.

We therefore copy a table from \cite{Henn} demonstrating subgroup inclusion. For each entry we include both the type and automorphism group, and, where relevant, we label the inclusions with the special values from the example forms of quartics provided in \autoref{thm:aut-gps-quartic}.

\[ \begin{tikzcd}
    \gpdiagram{I}{\PGL_2(7)} &  & \gpdiagram{II}{C_4^{\times2}\rtimes S_3} &  & \gpdiagram{III}{C_4\circledcirc A_4} & \gpdiagram{VI}{C_9}\\
     & \gpdiagram{IV}{S_4}\ar[ul,"a = \frac{3}{2}(-1 \pm \sqrt{-7})"]\ar[ur,"a=0"] &  &\gpdiagram{V}{P}\ar[ul,"{a=0,\pm6}"]\ar[ur,"a=\pm2\sqrt{-3}"]  & \gpdiagram{VIII}{C_6}\uar["a=0"] & \\
    \gpdiagram{IX}{S_3}\ar[ur,"a=0"] &  & \gpdiagram{VII}{D_8}\ar[ul,dotted]\ar[ur,"b=0"] &  &  & \\
     &  &  \gpdiagram{X}{K_4}\uar["\tiny\substack{\{\pm a,\pm b,\pm c\} \\ \text{not all} \\ \text{distinct}}"] &  & \gpdiagram{XI}{C_3}\ar[uu,"\substack{a=-1\\ \text{or} \\ b=-1}"]\ar[uuur, "\tiny\substack{a=\zeta_3 \\ \text{or} \\b=\zeta_3}" below right]  & \\
     &  &  \gpdiagram{XII}{C_2}\uar["b=0"]\ar[uull,"\tiny\substack{a=-2 \\ b=0 \\ c=-d}"]\ar[uuurr,dotted] &  & & \\
     &  &  & e\ar[ul]\ar[uur] &  & \\
\end{tikzcd} \]
The groups as we've presented them may not be literally subgroups of one another, a projective change of basis may be needed. We explain the meaning behind the dashed arrows in \autoref{rmk:dashed-arrows}.

\begin{example} Some lovely examples of the phenomenon we're discussing can be found in \cite[pp.~269--270]{Dolgachev}; we include one such example here, where we show how to derive the special values $a= -6,0,6$ where the pencil of Type V quartics becomes a Dyck quartic. Recall the equation for a Type V quartic:
\begin{align*}
   x^4 + y^4 + z^4 + ax^2y^2.
\end{align*}
It is clear that by setting $a=0$ we obtain the Dyck quartic, but where do the other values $a = \pm6$ come from? We see that the Pauli group $P$ is not literally a subgroup of $C_4^{\times 2 }\rtimes S_3$ as presented, nevertheless $C_4^{\times 2 }\rtimes S_3$ contains a unique conjugacy class of subgroups isomorphic to $P$. This is reflected via a change of basis matrix (unique up to conjugacy in $C_4^{\times2} \rtimes S_3$ --- note there are three conjugacy classes of $P$ in $C_4^{\times2} \rtimes S_3$)
\begin{align*}
    g &=  \begin{pmatrix} 1 & 1 & 0 \\
    1 & -1 & 0 \\
    0 & 0 & 8^{1/4}\end{pmatrix}.
\end{align*}
Conjugating $P$ by this change of basis, we see that $gPg^{-1} \le C_4^{\times 2} \rtimes S_3$, and the equation for a Type V quartic changes to be
\begin{align*}
   (x+y)^4 + (x-y)^4 + 8^{1/4}z^4 + a(x+y)^2(x-y)^2 = (a+2)x^4 + (a+2)y^4 + 8z^4 + (12-2a)x^2 y^2.
\end{align*}
We can now ask for which values of $a$ does such a quartic satisfy the $C_4^{\times r}\rtimes S_3$ symmetries in \autoref{eqn:type-II-grp}. The first generator cycles the variables, for which we see we need $12-2a=0$ for the $x^2 y^2$ term to vanish, i.e. we need $a=6$. For this choice of $a$ we satisfy the full symmetry group $C_4^{\times2}\rtimes S_3$ and hence obtain the Type II quartic. The case $a=-6$ is similar, but using a conjugate change of basis.
\end{example}

\begin{remark} 
We remark this line of reasoning is basically equivalent to being aware of clever relations for working with homogeneous polynomials and their vanishing loci. Dolgachev remarks that the constraint above comes from knowing the identity
\begin{align*}
    x^4 + y^4 &= \frac{1}{8} \left( (x+y)^4 + (x-y)^4 + 6(x+y)^2(x-y)^2 \right).
\end{align*}
Leveraging identities such as these in projective geometry is characteristic of much of algebra in the latter half of the 19th century.
\end{remark}

What happens when this line of reasoning fails? Consider the following example:

\begin{proposition}\label{prop:C2-C6} There are no special values of $a,b,c,d$ for which a Type XII quartic becomes a Type VIII quartic.
\end{proposition}
\begin{proof} Since $C_6$ is abelian and $C_2$ is a subgroup as written, we just have to check whether any special values of $a,b,c,d$ exist for which a Type XII quartic can satisfy a sixfold symmetry. Applying the generator for $C_6$ to a Type XII quartic, we obtain
\begin{align*}
    x^4 + y^4 + \zeta_3z^4 + x^2(ay^2 + b\zeta_3yz + c\zeta_3^2z^2) + d\zeta_3^2 y^2 z^2,
\end{align*}
and there are no values of $a,b,c,d$ for which this is equal to the original equation.\footnote{Note the existence of Type III quartics does not contradict this line of reasoning --- the special values for passing from XII to X and then from X to VII are contradictory, so an intermediate change of basis is first needed, and new parameters need to be defined.}
\end{proof}

\begin{remark}\label{rmk:dashed-arrows} This is precisely the meaning of the dashed arrows in our subgroup diagram. In the language of the moduli of symmetric quartics, this means the Type VIII quartics don't lie on the boundary of the Type XII quartics, but rather they share a common boundary at the Type III quartics. The argument for the dashed arrow from Type VII to Type IV is omitted.
\end{remark}

\subsection{Leveraging Euler classes to explore restricted actions}

We've seen that the locus of $D_8$-symmetric quartics is not glued to the locus of $S_4$-symmetric quartics, so from the naive perspective of the moduli space, we shouldn't expect that there is much coherence between how these groups act on bitangents. Nevertheless we can see the following example:

\begin{example}\label{exa:edge-d8} 
Consider the Edge quartic (of Type IV) with its restricted $D_8\le S_4$-symmetry on the $xy$-plane. The $D_8$-orbits of the lines are graphed in \autoref{fig:edge-d8}
\begin{figure}[h]
  \includegraphics[width=0.4\linewidth]{Edge_no_lines.png}
   \includegraphics[width=0.4\linewidth]{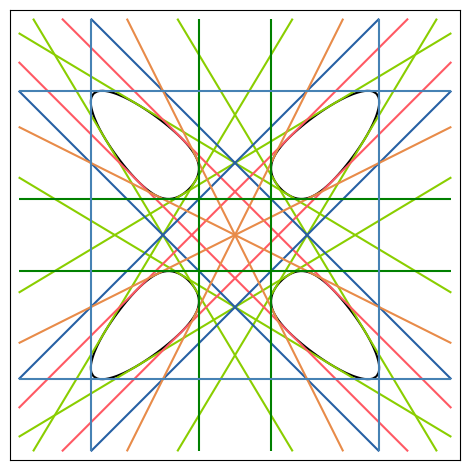}
  \centering
  \caption{The \textit{Edge quartic}, with $D_8$-orbits on its bitangents given by $[D_8/e] + 2[D_8/C_2^{(1)}] + 2[D_8/C_2^{(2)}] + [D_8/C_2^Z]$.}
  \label{fig:edge-d8}
\end{figure}
\end{example}

This is the same count as the $D_8$-action on the bitangents of a Type VII quartic (\autoref{comp:D8}), despite the fact that the Edge quartic cannot be obtained by deforming a family of quartics with automorphism group $D_8$. So the fact that we obtained the same answer above is perhaps not obvious from the moduli perspective. It is, however, obvious from the perspective of equivariant Euler classes!

\begin{proposition}\label{prop:conjugacy} 
Let $H \le G$ be a subgroup, and let $g \in G$ be arbitrary. Then we can consider restricting a $G$-set along the inclusion of $H$ in $G$, or along the group homomorphism $c_g \colon H \to G$ given by $h \mapsto ghg^{-1}$. These induce naturally isomorphic functors $\Fin_G \to \Fin_H$.
\end{proposition}
\begin{proof} It suffices to observe that conjugacy by $g$ is a natural isomorphism on the category of finite $G$-sets.
\end{proof}

\begin{corollary} Let $E \to M$ be any $G$-equivariant enumerative problem satisfying the hypotheses in \cite[1.1]{EEG} (e.g. counting lines on symmetric cubic surfaces or counting bitangents to planar quartics). Then the restricted Euler class $n_H(E)$, obtained as the image of
\begin{align*}
    \pi_0^G \MU_G &\to \pi_0^H \MU_H \\
    n_G(E) &\mapsto n_H(E)
\end{align*}
depends only on the conjugacy class of $H$ in $G$.
\end{corollary}
\begin{proof} Equivariant Euler classes valued in homotopical bordism are in the injective image of the Burnside ring by \cite[5.24]{EEG}, hence it suffices to prove the statement there. This follows directly from group completing isomorphism classes of finite $G$-sets and applying \autoref{prop:conjugacy}.
\end{proof}

This explains what we observed in \autoref{exa:edge-d8}. Although there we cannot deform Type VII quartics into Type IV quartics, there are unique conjugacy classes of $D_8$ and $S_4$ in $\PGL_3$ occurring as the automorphism groups of quartics, and the copy of $D_8$ is subconjugate to $S_4$.

\begin{remark} $\ $
\begin{enumerate}
    \item This means that we can still restrict computations along the dashed arrows in our subgroup lattice, despite the fact that the associated components in the moduli are not attached directly to one another.
    \item This line of reasoning fails for isomorphic subgroups which are not conjugate --- given an $S_4$-symmetric cubic surface, the restricted action on its lines via the two non-conjugate cyclic subgroups of order two are different, as computed in \cite[6.5]{EEG}.
\end{enumerate}

\end{remark}

\bibliographystyle{amsalpha}
\bibliography{citations.bib}{}
\end{document}